\documentclass{amsart}[11pt]
\usepackage{amsmath,amsfonts,amssymb,latexsym,amsthm}
\usepackage{graphicx,epsf}
\usepackage{color}
\usepackage{enumerate}
\usepackage{comment}

\newtheorem{thm}{Theorem}[section]
\newtheorem{lemma}[thm]{Lemma}

\newtheorem{prop}[thm]{Proposition}
\newtheorem{conj}[thm]{Conjecture}

\newtheorem{rem}[thm]{Remark}

\theoremstyle{definition}
\newtheorem{Example}[thm]{Example}
\newtheorem{Remark}[thm]{Remark}

\def\zz{\mathbb Z}

\def\gg{\mathbb G}

\def\ee{\mathbb E}

\def\ga{\gamma}

\def\ep{\ve}
\def\al{\alpha}
\def\be{\beta}
\def\om{\omega}

\def\ve{\varepsilon}

\def\cC{\mathcal C}

\def\T{\mathbf{T}}
\def\CT{\mathcal T}

\def\ssu{\subset}

\def\<{\langle}
\def\>{\rangle}

\def\ts{\hskip.015cm}

\def\0{{\mathbf 0}}
\def\id{{\rm \bf i}}

\newcommand{\ra}{\rightarrow}
\newcommand{\e}{\ep}

\newcommand{\lb}{\langle}
\newcommand{\rb}{\rangle}
\newcommand{\len}{\ell}
\newcommand{\deq}{\mathrel{\mathop:}=}

\newcommand{\pathto}[1][]{\:\rightsquigarrow^{#1}\:}
\newcommand{\edgeto}{\:\rightarrow\:}
\newcommand{\oaprod}{\mathop{\overrightarrow{\prod}}}

\newcommand{\GL}{GL}
\renewcommand{\i}{\id}
\newcommand{\Frat}{\Phi} %Frattini subgroup
\newcommand{\Tbin}{\mathbf{T}}
\newcommand{\AutT}{\Aut(\Tbin)}
\newcommand{\setc}[3][]{\left\{ #2 \,\middle|\, #3 \vphantom{#1|}\right\}}
\newcommand{\setcst}[3][]{\left\{#2\;\;\text{s.t.}\;#3 \vphantom{#1|}\right\}}
\newcommand{\gen}[1]{\left\lb #1 \right\rb}
\newcommand{\abs}[2][]{\left|#2 \vphantom{#1|}\right|}
\newcommand{\size}{\mathop{\#}}

\DeclareMathOperator{\Stab}{Stab}
\DeclareMathOperator{\supp}{supp}
\DeclareMathOperator{\Rist}{Rist}
\DeclareMathOperator{\Aut}{Aut}

\DeclareMathOperator{\da}{\!\downarrow}

\DeclareMathOperator{\ord}{ord}
\DeclareMathOperator{\Cay}{Cay}
\DeclareMathOperator{\Schr}{Schr}

\def\dist{{\text {\rm dist} } }

\begin{document}
\title{Growth in product replacement graphs of Grigorchuk groups}

\author[Anton~Malyshev]{ \ Anton~Malyshev$^\star$}
\author[Igor~Pak]{ \ Igor~Pak$^\star$}

\thanks{\thinspace ${\hspace{-.45ex}}^\star$Department of Mathematics,
UCLA, Los Angeles, CA, 90095.
\hskip.06cm
Email:
\hskip.06cm
\texttt{\{amalyshev,pak\}@math.ucla.edu}}

\begin{abstract}
The product replacement graph $\Gamma_k(G)$ is the graph on the
generating~$k$-tuples of a group~$G$, with edges corresponding to Nielsen moves. We prove the exponential growth of product replacement graphs $\Gamma_k(\gg_\omega)$ of Grigorchuk groups, for~$k \geq 5$.
\end{abstract}

\maketitle
\theoremstyle{plain}

\section{Introduction}

The \emph{product replacement graphs} $\Gamma_k(G)$ are the graphs on
generating $k$-tuples of a group~$G$, with edges corresponding to
multiplications of one generator by another (see below).
These graphs play an important role in computational group theory
(see e.g.~\cite{BL,NP,whatprp}), and are related to the Andrews--Curtis conjecture
in algebraic topology (see e.g.~\cite{BKM,BLM,Met}).
For infinite groups, proving non-amenability of graphs $\Gamma_k(G)$ is a
major open problem, closely related to Kazhdan's property~(T) of \ts $\Aut(F_k)$.
In this paper we establish a weaker property, the exponential growth
of product replacement graphs, for the Grigorchuk group $\gg$ and its generalizations $\gg_\omega$.

\smallskip

Let us begin by stating the main conjecture we address in this paper.
% The motivation behind it is postponed until final remarks
% (see Subsection~\ref{ss:fin-aut}).

\begin{conj}[Main Conjecture] \label{conj:main}
Let~$G$ be an infinite group generated by $d$ elements.  Then the
product replacement graphs $\Gamma_k(G)$ have exponential growth, for
all $k\ge d+1$.
\end{conj}

\medskip

\noindent
Formally speaking, graphs~$\Gamma_k(G)$ can be disconnected, in which
case we conjecture that \emph{at least one} connected component
has exponential growth.

\smallskip

The motivation behind our Main Conjecture is rather interesting,
which makes the conjecture both natural and speculative.  First,
recall that $\Gamma_k(G)$ are Schreier graphs of $\Aut(F_k)$, generated
by Nielsen transformations~\cite{LP} (see also~\cite{LZ,whatprp}).
A well known conjecture states that $\Aut(F_k)$ has \emph{Kazhdan's
property~$(T)$} \ts for $k>3$. If true, this would imply the following
conjecture:

\begin{conj}\label{conj:non-amenable}
For every infinite group $G$, product replacement graphs $\Gamma_k(G)$
are \emph{non-amenable}, for $k$ large enough.
\end{conj}

In particular, this conjecture implies that all connected components of~$\Gamma_k(G)$
are infinite and have exponential growth, for all~$G$ and~$k$ large enough.
We should mention that $\Aut(F_k)$ does not have~$(T)$ for $k=2$ and~$3$
(see~\cite{GL,lubotzky_book}).  On the other hand,  the non-amenability of
$\Gamma_n(G)$ follows from a weaker property~$(\tau)$ for an appropriate family
of subgroups (see~\cite{LZ}).

\smallskip

We approach the Main Conjecture by looking at the growth of groups.
The conjecture is straightforward for groups of exponential growth.
It can be shown that the conjecture holds for virtually nilpotent
groups (see Section~\ref{s:basic}). By Gromov's theorem, this implies
the conjecture for all groups of polynomial growth.

Unfortunately, groups of intermediate growth lack the rigid
structure of nilpotent groups, so much that even explicit
examples are difficult to construct and analyze
(see e.g.~\cite{de_la_harpe,Gri3}).  Even now, much remains
open for the classical \emph{Grigorchuk group}~$\gg$, the first
example of a group of intermediate growth discovered by
Grigorchuk (see~\cite{grigorchuk_growth,solved_unsolved}).

We present a new combinatorial technique which allows us to establish
the conjecture for a large class of Grigorchuk groups~$\gg_\om$.
This is the main result of this paper:

\begin{thm}\label{thm:gg_fast}
Let $\gg_\omega$ be a generalized Grigorchuk group. Then $\Gamma_n(\gg_\omega)$
is connected for each $n \geq 4$, and has exponential growth for each $n \geq 5$.
\end{thm}

The techniques in this paper generalize fairly easily to several other
groups of intermediate growth, such as the Gupta--Sidki $p$-groups~\cite{gupta-sidki},
as well as large families of Grigorchuk $p$-groups. Many groups of intermediate
growth, such as the groups of oscillating growth defined in~\cite{kassabov_pak}, have, by
construction, some $\gg_\omega$ as a subgroup or a factor group. Such groups, then, also have
exponential Nielsen growth (see Proposition~\ref{prop:transfer_fast}).

In fact, the techniques in this paper apply to a general class of branch groups
defined in~\cite{bartholdi} called \emph{splitter-mixer groups}.
Many known group of intermediate growth appears to be based
on a splitter-mixer group. (An example that does not fall
into this class is given in~\cite{Nek} and analyzed in~\cite{monodromy},
but our techniques should apply there as well). The proofs will appear in~\cite{thesis}.

In summary, although we have yet to find proofs in all cases, we believe the Main Conjecture holds
for all \emph{known} constructions of groups of intermediate growth.  In that sense the situation
is similar to the ``$p_c < 1$'' conjecture by Benjamini and Schramm~\cite{BS} for groups of
superlinear growth.  The conjecture is known to hold for groups of exponential and polynomial
growth, and by an ad hoc argument for Grigorchuk groups and general self-similar groups~\cite{MP}.
It remains open for general groups of intermediate growth (see~\cite{Pete}).

% This technique will be extended to a large class groups in~\cite{thesis},
% to include many known examples of groups of intermediate growth.
%(see Subsection~\ref{ss:fin-gen}).

Let us mention that in a followup paper~\cite{Mal}, the first author establishes
Conjecture~\ref{conj:non-amenable} for several classes of groups of exponential
growth, which include virtually solvable groups, linear groups,
random finitely presented groups (in Gromov sense), and hyperbolic groups.
He uses a technical extension of \emph{uniform exponential growth}
and \emph{uniform non-amenability} (see~\cite{uniform_nonamenability,BG,dlh2,Wil}).

Unfortunately, the explicit combinatorial approach in this paper,
does not seem to be strong enough to establish Conjecture~\ref{conj:non-amenable}
for the Grigorchuk group, which we state as a separate conjecture of independent
interest.

\begin{conj}\label{conj:grig-nonam}
Product replacement graphs $\Gamma_k(\gg)$ are non-amenable, for all $k\ge 5$.
\end{conj}

\smallskip

The rest of this paper is structured as follows.  We begin with
basic definitions of growth of groups and the
product replacement graphs
(Section~\ref{s:def}).  In Section~\ref{s:basic} we present basic
results on the growth and connectivity of graphs $\Gamma_k(G)$; we
also present general tools for establishing the
exponential growth results. In a technical Section~\ref{s:binary}
we describe general tools and techniques for working with subgroups
$G \ssu \Aut(\mathbf{T}_2)$ and their product replacement graphs.
In the next two sections~\ref{sec:grigorchuk} and~\ref{s:gen_grig}
we establish the main result.  First, we prove the exponential
growth of $\Gamma_k(\gg)$ for $k\ge 5$; in this case the (technical)
argument is the most lucid.  We then generalize this approach to
\emph{all} Grigorchuk groups~$\gg_\om$.  We conclude with final
remarks and open problems (Section~\ref{s:fin}).

\bigskip

%%%%%%%%%%%%%%%%%%%%%%%%%%%%%%%%%%%%%%%%%%%%%%%%%%%%%%%%%%%%%%%%%%

\section{Background and definitions}\label{s:def}

\subsection{Notation}

Let $X$ be a finite set. We write $\size{X}$ or $\abs{X}$ to denote the size of $X$. Throughout the paper we use $\zz_n$ to denote the cyclic group $\zz/n\zz$.

Let $\Gamma$ be a directed graph, which may have loops and repeated edges. We define $v \in \Gamma$ to mean that $v$ is a vertex of $\Gamma$. Let $v, w$ be vertices of~$\Gamma$. We write $v \edgeto w$ when there is an edge in $\Gamma$ from $v$~to~$w$, and $v \pathto w$ when there is a path in $\Gamma$ from $v$~to~$w$. We say $\Gamma$ is \emph{symmetric} if for every edge $v \edgeto w$ of $\Gamma$ there is an \emph{inverse edge} $w \edgeto v$. Every graph considered in this paper is a symmetric directed graph, unless otherwise specified. When convenient, we think of a symmetric directed graph as an undirected graph by identifying every edge with its inverse.

Let $G$ be a group, which may be finite or infinite. A \emph{generating $n$-tuple} of $G$ is an element $(g_1, \dots, g_n) \in G^n$, such that $G = \gen{g_1, \dots, g_n}$. Let $S = (g_1, \dots, g_n)$ be such an $n$-tuple. Consider a left action of $G$ on a set $X$. The \emph{Schreier graph} $\Schr_S(G,X)$ of this action with respect to $S$, is the directed graph whose vertices are the elements of $X$, with edges $x \edgeto g_ix$ and $x \edgeto g_i^{-1}x$ for each $x \in X$, and each $0 \leq i \leq n$. Note that each vertex in $\Schr_S(G,X)$ has $2n$ edges leaving it, and each edge $v \edgeto w$ in such a graph has an inverse edge $w \edgeto v$. Thus, $\Schr_S(G,X)$ is a $2n$-regular symmetric directed graph.

The \emph{Cayley graph} $\Cay_{S}(G)$ is the Schreier graph $\Schr_S(G,G)$ with respect to the left action of $G$ on itself by multiplication. Clearly, the Cayley graph $\Cay_{S}(G)$ is connected. Given $g \in G$, we define $\len_S(g)$ to be the length of the shortest path from $1$ to $g$ in the Cayley graph of $G$.

When the context makes it clear what the generating $n$-tuple $S$ is, we drop the subscript, and simply write $\Cay(G)$, $\Schr(G,X)$, and~$\len(g)$. We write $\Aut(G)$ for the group of automorphisms of $G$. We write $H < G$ when $H$ is a subgroup of~$G$, and $H \lneqq G$ when $H$ is a proper subgroup of~$G$. For an element~$g \in G$, denote by~$\ord(g)$ the order of~$g$. For $g_1, \dots, g_n \in G$, denote
\begin{align*}
\oaprod_{i=1\dots n} g_i \, = \, g_1 \cdots g_n.
\end{align*}

\subsection{Growth in graphs}

Let $\Gamma$ be a symmetric directed graph, and let $v \in \Gamma$. The ball of radius $r$ centered at $v$, denoted $B_\Gamma(v,r)$, is the set of vertices $w \in \Gamma$ such that there is a path of length at most $r$ between $v$~and~$w$. For example, suppose $\Gamma = \Cay_{S}(G)$. Then $B_\Gamma(1,r)$ consists of the elements $g \in G$ for which $\len_S(g) \leq r$.

We say $\Gamma$ has \emph{exponential growth from $v$}, if there is a constant $\al>1$, such that $\abs{B_\Gamma(v,r)} \geq \al^r$ for all $r$ (equivalently, for sufficiently large $r$). Suppose $\Gamma$ has exponential growth from $w$, and there is a path $v \pathto w$ in $\Gamma$. Then $\Gamma$ also has exponential growth from $v$. Thus, if $\Gamma$ is connected and has exponential growth from some $v \in \Gamma$, it also has exponential growth from any $w \in \Gamma$. In this case, we say that $\Gamma$ has \emph{exponential growth}.

\subsection{Growth in groups}

Let $G$ be a group, Let $S$ be a generating $n$-tuple of~$G$. Define~$B_{G,S}(r) = B_\Gamma(1,r)$, where ${\Gamma = \Cay_{S}(G)}$. When it is clear what $S$ is, we simply write~$B_G(r)$ instead. It is easy to verify that the following definitions are independent of the choice of generators $S$.

We say $G$ has \emph{exponential growth} if $\Gamma$ has exponential growth. In other words, $G$ has exponential growth if there is a constant $\al > 1$ such that $\abs{B_G(r)} \geq \al^r$ for sufficiently large $r$. Equivalently  $G$ has exponential growth if and only if
\begin{align*}
\liminf_{r\ra \infty} \, \frac{\log\abs[\big]{B_G(r)}}{r} &
> 0.
\end{align*}

Similarly, we say $G$ has \emph{polynomial growth} if there is a constant $d$ with $\abs{B_G(r)} \leq r^d$ for sufficiently large~$r$. In other words, $G$ has polynomial growth if
\begin{align*}
\limsup_{r\ra \infty} \, \frac{\log \abs[\big]{B_G(r)}}{\log r} &
< \infty.
\end{align*}

\begin{Example}
The group $\zz$ has polynomial growth. With respect to the generating $1$-tuple $S = (1)$, we have $B_{\zz}(r) = [-r, r]$, and hence $\abs{B_{\zz}(r)} = 2r + 1$.
\end{Example}

\begin{Example}
The free group with two generators, $G = F_2 = \gen{a,b}$ has exponential growth. With respect to the generators $S = (a,b)$, we have $\abs{B_{G}(r)} = 1+4\cdot 3^{r-1}$ for~$r \geq 1$.
\end{Example}

We say $G$ has \emph{intermediate growth} if it has neither exponential nor polynomial growth. The first known example of a group of intermediate growth is the Grigorchuk group~$\gg$, which will be defined later, in Section~\ref{sec:grigorchuk}. We refer to~\cite[$\S$VI]{de_la_harpe} and~\cite{growthintro} for more on the growth of groups. % (see also~$\S$\ref{ss:fin-gen}).

\subsection{Product replacement graphs}

Given a generating $n$-tuple of $S$ a group $G$, we can take an element of $S$ and multiply it, either on the left or the right, by another element or another element's inverse. Such an operation is called a \emph{Nielsen move}. Formally, for each $1 \leq i,j \leq n$ with $i \neq j$, we define the Nielsen moves $R_{ij}^{\pm 1}$, $L_{ij}^{\pm1}$ by
\begin{align*}
R_{ij}^{\pm 1}(g_1, \dots, g_i, \dots, g_j, \dots g_n) & = (g_1, \dots, g_i, \dots, g_jg_i^{\pm 1}, \dots, g_n),
\\
\text{ and } \quad
L_{ij}^{\pm 1}(g_1, \dots, g_i, \dots, g_j, \dots g_n) & = (g_1, \dots, g_i, \dots, g_i^{\pm 1} g_j, \dots, g_n).
\end{align*}
Clearly, if $S$ is a generating $n$-tuple of $G$, then $R_{ij}S$, $R_{ij}^{-1}S$, $L_{ij}S$, and~$L_{ij}^{-1}S$ are also generating $n$-tuples of $G$.

We define the \emph{product replacement graph} $\Gamma_n(G)$ to be the directed graph whose vertices are the generating $n$-tuples of $G$, where there is an edge from $S$ to $R_{ij}S$, $R_{ij}^{-1}S$, $L_{ij}S$, and $L_{ij}^{-1}S$, for each generating $n$-tuple $S$ and each pair of integers $i \neq j$ satisfying~$1 \leq i,j \leq n$. This is a $4n(n-1)$-regular symmetric directed graph.

Observe that
\begin{align*}
R_{ij}L_{ji}^{-1}L_{ij}(g_1, \dots, g_i, \dots, g_j ,\dots, g_n)
\, &= \,
R_{ij}L_{ji}^{-1}(g_1, \dots, g_i, \dots, g_i g_j ,\dots, g_n)
\\ = \,
R_{ij}(g_1, \dots, g_j^{-1}, \dots, g_i g_j ,\dots, g_n) \, &= \,
(g_1, \dots, g_j^{-1}, \dots, g_i ,\dots, g_n).
\end{align*}
Hence, a series of Nielsen moves can swap two elements in a generating $n$-tuple, inverting one of them. Doing this twice simply inverts both elements. This implies that Nielsen moves permit us to rearrange generators in an $n$-tuple, except that we may need to invert one element (see \cite{whatprp}). Moreover, if $g_i = 1$ for some $i$, then we can use Nielsen moves invert any one element, and therefore we can rearrange the generators freely.

\begin{Example}\label{ex:Gamma_2_zz}
The graph $\Gamma_2(\zz)$ has a vertex for each pair of relatively prime integers $(a,b)$,
with two edges from $(a,b)$ to each of $(a,b+a)$, $(a,b-a)$, $(a+b,b)$ and~$(a-b,b)$.
It is easy to check that this graph has exponential growth: the subgraph induced by $\setc{ (a,b) \in \zz^2 }{a, b > 0,\;\gcd(a,b) = 1}$ is a rooted binary tree.
\end{Example}

\begin{Example}
Let $G = \zz_p^n$, with $p$ prime. Then $\Gamma_n(G)$ is the set of bases of $\zz_p^n$ as a vector space over~$\zz_p$. These bases are in one-to-one correspondence with matrices in $\GL_n(\zz_p)$, and Nielsen moves correspond to elementary row operations. Row operations do not change the determinant of a matrix. It follows that there is one connected component for every value of the determinant. This implies that $\Gamma_n(\zz_p^n)$ has $p-1$ connected components (see \cite{diaconis1999graph}).
\end{Example}

\subsection{Growth of $\Gamma_n(G)$}

Let $S = (g_1, \dots, g_n) \in \Gamma_n(G)$. We write
\begin{align*}
S^{(m)}
\deq
(g_1, \dots, g_n,1,\dots,1) \in \Gamma_{n+m}(G),
\end{align*}
and define $\Gamma_{n+m}(G,S)$ to be the connected component of $\Gamma_{n+m}(G)$ containing~$S^{(m)}$.

We say $G$ has \emph{exponential Nielsen growth} if $\Gamma_n(G,S)$ has exponential growth for some $n$ and some generating $n$-tuple $S$ of $G$.
It is easy to show that a finitely generated group~$G$ has exponential Nielsen growth if $G$ is either an infinite group of polynomial growth, or a group of exponential growth (see~Proposition~\ref{prop:poly_exp_fast}). This suggests that every infinite finitely generated group has exponential Nielsen growth:

\begin{conj}\label{conj:inf_fast}
For every infinite finitely generated group $G$, there is an generating $n$-tuple $S \in \Gamma_n(G)$ such that $\Gamma_n(G,S)$ has exponential growth.
\end{conj}

Note that this conjecture is a weaker version of Conjecture~\ref{conj:non-amenable}.
Here we accounted for the possibility that there can be many connected components, and are working with only
one of them.  Our Main Conjecture~\ref{conj:main} is also stronger; implicit in it is a reference to a
conjecture that every generating $k$-tuple is connected to a redundant generating $k$-tuple in $\Gamma_k(G)$.
For this and stronger conjectures on connectivity of $\Gamma_k(G)$, see~\cite{whatprp} (see also~\cite{BKM}).

\medskip

\section{Basic results}\label{s:basic}

\subsection{Growth of graphs}
We do not need to prove that $B_\Gamma(v,r)$ is large for every single $r$ to conclude that $\Gamma$ has exponential growth from~$v$. As the following lemma shows, it suffices to prove it for a relatively sparse set of numbers~$r$.

A sequence of positive integers $r_1, r_2, \dots$ is called \emph{log-dense} if it is increasing, and there is a constant $\be$ such that $r_{i+1} \leq \be r_i$ for every~$i\geq 1$. In other words, an increasing integer sequence $(r_i)$ is log-dense if the gaps in the sequence $(\log r_i)$ are bounded above.

\begin{lemma}\label{lem:sparse_exp}
Let $\Gamma$ be a symmetric directed graph, and let~$v$ be a vertex of $\Gamma$. Suppose that for some constant $\al > 1$, there is a log-dense sequence $r_1, r_2, \dots$ such that~$\abs{B(v,r_i)} \geq \al^{r_i}$ for every $i \geq 1$. Then $\Gamma$ has exponential growth from~$v$.
\end{lemma}
\begin{proof}
Since $r_i$ is an increasing sequence of positive integers, we can conclude that for sufficiently large~$r$, there is an $i$ with~$r_i \leq r \leq r_{i+1}$. Since $r_{i+1} \leq \be r_i$, we have~$r_i \geq r/\be$. Thus,
\begin{align*}
\abs{B(v,r)} &
\geq \abs{B(v,r_i)}
\geq \al^{r_i}
\geq \al^{r/\be},
\end{align*}
which implies the result.
\end{proof}

If a graph $\Gamma$ is a covering of another graph $\Gamma'$, and $\Gamma'$ has exponential growth, then so does~$\Gamma$.

\begin{prop}\label{prop:graph_transfer_covering} Let $\Gamma'$ and $\Gamma$ be symmetric directed graphs, and suppose $\phi: \Gamma' \ra \Gamma$ maps the set of neighbors of each vertex $v \in \Gamma'$ surjectively onto the neighbors of~$\phi(v)$. Suppose $\Gamma$ has exponential growth from~$\phi(w)$. Then $\Gamma'$ has exponential growth from~$w$.
\end{prop}
\begin{proof}It suffices to show that $\phi$ maps $B_{\Gamma'}(w,r)$ onto~$B_{\Gamma}(\phi(w),r)$ for all $r\geq 0$, since in that case~
\begin{align*}
\abs{B_{\Gamma'}(w,r)} &
\geq \abs{B_{\Gamma}(\phi(w),r)}.
\end{align*}
We prove this by induction on $r$. The base case $r = 0$ is trivial. Suppose
\begin{align*}
\phi\big(B_{\Gamma'}(w,r)\big) \supseteq B_{\Gamma}\big(\phi(w),r\big),
\end{align*}
and consider~$v \in B_{\Gamma}(\phi(w),r+1)$. We know that $v$ has a neighbor $u \in B_{\Gamma}(\phi(w),r)$, which has a preimage~$u' \in B_{\Gamma'}(w,r)$. Since $v$ is a neighbor of $u$, we know that some neighbor of~$u'$ is mapped to~$v$. Therefore, $v \in \phi\big(B_{\Gamma'}(w,r+1)\big)$, as desired.
\end{proof}

It is easy to see that if a graph $\Gamma$ is a subgraph of $\Gamma'$, and $\Gamma$ has exponential growth, so does $\Gamma'$. Moreover, we have the following stronger result:

\begin{prop} \label{prop:graph_transfer_near-injective}
Let $\Gamma$ and $\Gamma'$ be symmetric directed graphs, and suppose $\phi: \Gamma \ra \Gamma'$ sends neighbors to neighbors. Suppose that there is a constant $C$ such that $\size{\phi^{-1}(v')} \leq C$ for every vertex~$v' \in \Gamma'$. Suppose that $\Gamma$ has exponential growth from~$w$. Then $\Gamma'$ has exponential growth from~$\phi(w)$.
\end{prop}
\begin{proof}It suffices to show that $\phi$ maps $B_\Gamma(w,r)$ into $B_{\Gamma'}(\phi(w),r)$ for all $r \geq 0$, since in that case
\begin{align*}
\abs{B_{\Gamma'}(\phi(w),r)}
&\geq\abs{B_{\Gamma}(w,r)}/C.
\end{align*}
We prove this by induction or~$r$. The base case $r = 0$ is trivial. Suppose
\begin{align*}
\phi\big(B_\Gamma(w,r)\big) &\subseteq B_{\Gamma'}\big(\phi(w),r\big),
\end{align*}
and consider $v \in B_\Gamma(w,r+1)$. We know that $v$ has a neighbor $u \in B_\Gamma(w,r)$, and~$\phi(u) \in B_{\Gamma'}(\phi(w),r)$. Since $u$ and $v$ are neighbors, and $\phi$ sends neighbors to neighbors, we see that $\phi(v)$ is a neighbor of~$\phi(u)$. It follows that ${\phi(v) \in B_{\Gamma'}(\phi(w),r+1)}$, as desired.
\end{proof}

\subsection{Growth of product replacement graphs}

Observe that if $m \geq n$ then $\Gamma_n(G,S)$ embeds into $\Gamma_m(G,S)$. Therefore, by Lemma~\ref{prop:graph_transfer_near-injective} if $\Gamma_n(G,S)$ has exponential growth, so does $\Gamma_m(G,S)$.

Moreover, if $H$ is a finitely generated subgroup of $G$, then every product replacement graph of $H$ embeds in some product replacement graph of $G$. We can conclude that if a subgroup of $G$ has a product replacement graph of exponential growth, so does $G$. Formally:

\begin{prop}\label{prop:supergroup_fast}
Let $H$ and $G$ be finitely generated groups with~${H < G}$. Suppose some connected component of $\Gamma_m(H)$ has exponential growth, and let $S \in \Gamma_n(G)$. Then $\Gamma_{n+m}(G,S)$ has exponential growth. In particular, if $H < G$ and $H$ has exponential Nielsen growth, then $G$ also has exponential Nielsen growth.
\end{prop}
\begin{proof}
Let~$S = (g_1, \dots, g_n)\in \Gamma_n(G)$. We know that $\Gamma_m(H)$ has exponential growth from some $T \in \Gamma_m(H)$. Let $T = (h_1, \dots, h_m)$. There is a graph embedding $\phi: \Gamma_m(H)\ra\Gamma_{n+m}(G)$ given by
\begin{align*}
\phi(h_1', \dots, h_m') = ( g_1, \dots, g_n, h_1', \dots, h_m').
\end{align*}
Hence, $\Gamma_{n+m}(G)$ has exponential growth from~$\phi(T)$. Since the $g_i$'s generate $G$, we know that each $h_i$ is a product of $g_i$'s and their inverses. Thus, there is a sequence of Nielsen moves $S^{(m)} \pathto \phi(T)$, where
\begin{align*}
S^{(m)} &= (g_1, \dots, g_n, 1, \dots, 1),
\text{ and }
\phi(T) = (g_1, \dots, g_n, h_1, \dots, h_m).
\end{align*}
Therefore, $\Gamma_{n+m}(G,S) = \Gamma_{n+m}\big(G,\phi(T)\big)$, which implies that $\Gamma_{n+m}(G,S)$ has exponential growth.
\end{proof}

Similarly, we can show that if a group quotient of $G$ has a product replacement graph of exponential growth, then so does $G$.

\begin{prop}\label{prop:lift_fast}
Let $G$ and $H$ be finitely generated groups, and let $f:G \ra H$ be a surjective group homomorphism. Let $S\in \Gamma_n(G)$. Then the following hold.
\begin{enumerate}[{\normalfont (1)}]
\item \label{prop:lift_fast:S} Suppose $\Gamma_n\big(H,f(S)\big)$ has exponential growth. Then $\Gamma_n(G,S)$ has exponential growth.
\item \label{prop:lift_fast:S+} Suppose some connected component of $\Gamma_m(H)$ has exponential growth. Then $\Gamma_{n+m}(G,S)$ has exponential growth.
\item \label{prop:lift_fast:} Suppose $H$ has exponential Nielsen growth. Then $G$ also has exponential Nielsen growth.
\end{enumerate}
\end{prop}
\begin{proof}For (\ref{prop:lift_fast:S}), we extend $f$ to a map $\Gamma_n(G) \ra \Gamma_n(H)$ by making the following definition.
\begin{align*}
f(g_1, \dots, g_n) = \big(f(g_1), \dots, f(g_h)\big).
\end{align*}
This map $f$ sends the neighbors of every $T\in \Gamma_n(G)$ surjectively onto the neighbors of $f(T)$. Thus, since $\Gamma_n(H)$ has exponential growth from $f(S)$, we can apply Proposition~\ref{prop:graph_transfer_covering}, and conclude that $\Gamma_n(G)$ has exponential growth from $S$.

For (\ref{prop:lift_fast:S+}), let $S = (g_1, \dots, g_n) \in \Gamma_n(G)$, and choose
\begin{align*}
T &
= (h_1, \dots, h_m)
= \big(f(\tilde h_1), \dots, f(\tilde h_m)\big) \in \Gamma_m(H)
\end{align*}
such that $\Gamma_m(H,T)$ has exponential growth. Then
\begin{align*}
\Gamma_{n+m}\big(H,(f(g_1), \dots, f(g_n), h_1, \dots, h_m)\big)
\end{align*}
also has exponential growth. Thus, by (\ref{prop:lift_fast:S}),
\begin{align*}
\Gamma_{n+m}\big(G,(g_1, \dots, g_n, \tilde h_1, \dots, \tilde h_m)\big)
\end{align*}
has exponential growth. Since the $g_i$'s generate $G$, we know that there is a path in $\Gamma_{n+m}(G)$
\begin{align*}
(g_1, \dots, g_n, \tilde h_1, \dots, \tilde h_m)
\pathto
(g_1, \dots, g_n,1,\dots,1) = S^{(m)}.
\end{align*}
Hence, $\Gamma_{n+m}(G)$ also has exponential growth from $S^{(m)}$, i.e.~$\Gamma_{n+m}(G,S)$ has exponential growth.
Finally, part~(\ref{prop:lift_fast:}) follows immediately from~(\ref{prop:lift_fast:S+}).
\end{proof}

In a different direction, if $G$ has a product replacement graph of exponential growth, so does every quotient of $H$ by a finite subgroup.

\begin{prop}\label{prop:finite_quotient_fast}
Let $G$ and $H$ be finitely generated groups, and let $f:G \ra H$ be a surjective group homomorphism with finite kernel. For every $S \in \Gamma_n(G)$, if $\,\Gamma_n(G,S)$ has exponential growth, then $\,\Gamma_n\big(H,f(S)\big)$ has exponential growth. In particular, if $G$ has exponential Nielsen growth, then  $H$ also has exponential Nielsen growth.
\end{prop}
\begin{proof}
We extend the map $f:G \ra H$, to the map $f:\Gamma_n(G) \ra \Gamma_n(H)$, given by
\begin{align*}
f(g_1, \dots, g_n) = \big(f(g_1), \dots, f(g_h)\big).
\end{align*}
This map sends neighbors to neighbors, and the preimage of each vertex has bounded size. The graph $\Gamma_n(G)$ has exponential growth from~$S$. Hence, by Proposition~\ref{prop:graph_transfer_near-injective}, $\Gamma_n(H)$ has exponential growth from~$f(S)$.
\end{proof}

We summarize the previous three results in the following proposition.
\begin{prop}\label{prop:transfer_fast}
Let $G$ and $G'$ be finitely generated groups, and suppose $G$ is a subgroup, quotient, or extension by a finite group of~$G'$. If $G$ has exponential Nielsen growth, then~$G'$ also has exponential Nielsen growth.
\end{prop}

\begin{rem}{\rm
Proposition~\ref{prop:transfer_fast} relates the Nielsen growth of a subgroup $H$ of~$G$
to the Nielsen growth of $G$. We conjecture that for any finite index subgroup $H$ of $G$, if $\Gamma_n(G)$ has exponential growth,
then so does $\Gamma_k(H)$ of $G$, for sufficiently large $k$.  This would imply that the property
of having exponential Nielsen growth respects virtual isomorphism.  More generally,
it would be interesting to see if this property is an invariant under quasi-isometry.}
\end{rem}

The proposition gives us an easy way to prove that a fairly large class of groups have exponential Nielsen growth.

\begin{lemma}\label{lem:infinite_order_fast}
Let $G$ be a finitely generated group. Suppose~$G$ contains an element of infinite order. For every $S \in \Gamma_n(G)$ and every $m \geq n+2$, we have that $\Gamma_m(G,S)$ has exponential growth.
\end{lemma}

\begin{proof}
By assumption, the group $G$ contains a subgroup isomorphic to $\zz$. It is easy to see that $\Gamma_2(\zz)$ has exponential growth (see Example~\ref{ex:Gamma_2_zz}). By Proposition~\ref{prop:supergroup_fast}, it follows that $\Gamma_{n+2}(G,S)$ has exponential growth, and hence so does $\Gamma_m(G,S)$ for every $m \geq n+2$.
\end{proof}

In particular, we can prove that groups of polynomial or exponential growth all have exponential Nielsen growth, which leaves Conjecture \ref{conj:inf_fast} open only for groups of intermediate growth.
\begin{prop}\label{prop:poly_exp_fast}
Let $G$ be an infinite finitely generated group. Suppose that either $G$ has polynomial
or exponential growth. Then $G$ has exponential Nielsen growth.
\end{prop}
\begin{proof}
Suppose $G$ has polynomial growth. By Gromov's theorem, $G$ is virtually nilpotent \cite{gromov}. It follows that some subgroup of $G$ has infinite abelianization. Thus, $G$ has an element of infinite order and, by Lemma \ref{lem:infinite_order_fast}, $G$ has exponential Nielsen growth.

Now suppose $G$ has exponential growth. Let $S = (g_1, \dots, g_n)$ be a generating $n$-tuple of $G$ and denote $\Gamma = \Gamma_{n+1}(G,S)$. Let $r$ be any positive integer. For any $g \in B_{G,S}(r)$, the distance between $S^{(1)} = (g_1, \dots, g_n, 1)$ and $(g_1, \dots, g_n, g)$ in $\Gamma$ is at most $r$, i.e. $(g_1, \dots, g_n, g) \in B_\Gamma(S,r)$. Thus,
\begin{align*}
\abs{B_\Gamma(S,r)} \geq \abs{B_{G,S}(r)}.
\end{align*}
But $\abs{B_{G,S}(r)}$ grows exponentially in $r$, and thus so does $\abs{B_\Gamma(S,r)}$. That is, $\Gamma = \Gamma_{n+1}(G,S)$ has exponential growth, and therefore $G$ has exponential Nielsen growth.
\end{proof}

\begin{rem}{\rm
The Grigorchuk group $\gg$ does not have an element of infinite order, so Lemma~\ref{lem:infinite_order_fast} is not enough to show that its product replacement graphs have exponential growth. It can be shown that $\Gamma_n(G)$ has exponential growth for sufficiently large $n$ as long as there are elements of $G$ whose order is exponential in their word length (see~\cite{thesis}). The Grigorchuk group $\gg$ does not satisfy this condition either, but some of the generalized Grigochuk groups $\gg_\omega$ do.}
\end{rem}

\subsection{Effective results}

The Grigorchuk group has no elements of infinite order, so Lemma~\ref{lem:infinite_order_fast} is not strong enough to prove it has exponential Nielsen growth. We use a different approach. It is enough to find large cubes in $G$, as follows.

Let $G$ be any group, and let $(g_1, \dots, g_k) \in G^k$, we say the \emph{cube spanned by $(g_1, \dots, g_k)$} is
\begin{align*}
\cC(g_1, \dots, g_n) \deq \setc[\big]{ g_1^{\e_1} \cdots g_n^{\e_n} }{ \e_i \in \{0, 1\} }.
\end{align*}
Observe that $\size{\cC(g_1, \dots, g_n)} \leq 2^n$. We say $(g_1, \dots, g_k)$ is a \emph{cubic} $k$-tuple if
\begin{align*}
\size{\cC(g_1, \dots, g_k)} &
= 2^k.
\end{align*}

\begin{lemma}\label{lem:cubic_fast}
Let $G$ be a finitely generated group, and fix~$S \in \Gamma_n(G)$. Let $\al>1$ be a constant, and $(k_i)$ be a log-dense sequence. Suppose for each $i\geq 1$, there is a path $\ga$ of length at most $\al k_i$ in $\Gamma_{n}(G)$, such that $\gamma$ starts at $S$ and visits some $S_1, \dots, S_{k_i} \in \Gamma_n(G)$ in that order. Suppose further that there is a cubic $k_i$-tuple $(g_1, \dots, g_{k_i})$, where $g_j \in S_j$ for each $1 \leq j \leq k_i$. Then $\Gamma_m(G,S)$ has exponential growth for every $m \geq n+1$.
\end{lemma}
\begin{proof}
It is enough to show that $\Gamma_{n+1}(G,S)$ has exponential growth. Let $\Gamma = \Gamma_{n+1}(G)$, and~$k = k_i$. By Lemma~\ref{lem:sparse_exp}, it suffices to show that
\begin{align*}
\abs{B_\Gamma(S^{(1)}, (\al+1)k)} &
\geq 2^{k}.
\end{align*}
Given $(\e_1, \dots, \e_k) \in \{0,1\}^k$, we traverse the path $\gamma$ in the first $n$ coordinates of $\Gamma_{n+1}(G)$, but when we reach $S_j$, if $\e_j = 1$ we also apply a Nielsen transformation to multiply the last entry by~$g_j$. This gives us a path $\gamma'$ in $\Gamma_{n+1}(G)$ of length at most $\al k + k$. The path $\gamma'$ ends at an element of $\Gamma_{n+1}(G)$ whose last entry is~$g_1^{\e_1} \dots g_n^{\e_n}$. Since $(g_1, \dots, g_k)$ is cubic, there are $2^k$ distinct such elements. Thus, we have constructed $2^k$ distinct elements of $B_\Gamma(S^{(1)},\al k+k)$, as desired.
\end{proof}

\subsection{Connectivity of product replacement graphs}

Recall the \emph{Frattini subgroup} $\Frat(G)$,
\begin{align*}
\Frat(G) = \setc[\big]{g \in G}{\text{if $H \lneqq G$, then $\gen{H, g} \lneqq G$}}.
\end{align*}

\noindent
(see e.g.~\cite[$\S 10.4$]{hall_theory}). It is easy to see that $\Frat(G)$ is a normal subgroup of~$G$.  We need the following
connectivity result by Evans (see \cite[Theorem~4.3]{Eva}).

\begin{thm}[Evans]
\label{lem:frattini_connected}
Suppose $G$ is generated by some $n$-tuple. Let $m \geq n+1$, and suppose $\Gamma_m\big(G/\Frat(G)\big)$ is connected. Then $\Gamma_m(G)$ is connected.
\end{thm}

It is known that for any finite abelian group $G$ with $n$ generators, the product replacement graph $\Gamma_m(G)$ is connected for every $m > n$ \cite{diaconis1999graph} (see also \cite{whatprp}). We use only the following special case, which is easy to verify by hand.
\begin{lemma}\label{lem:(Z_2)^n-connected}
The product replacement graph $\Gamma_m(\zz_2^n)$ is connected for every~$m \geq n$.
\end{lemma}

In particular, suppose $G/\Phi(G) \cong \zz_2^n$. Then $\Gamma_{m}(G)$ is connected for every~$m > n$.

\begin{rem}{\rm
Theorem~\ref{lem:frattini_connected} is an analogue for infinite groups of the following result in~\cite{LP}
(see also~\cite{whatprp}).  Let~$G$ and $H$ be finite groups with $k$ generators, and $f:G \ra H$ is a
surjective group homomorphism, then the extension $f:\Gamma_k(G) \ra \Gamma_k(H)$ is surjective.
That is, every generating $k$-tuple of $H$ lifts to a generating $k$-tuple of~$G$.  As a corollary,
if $\Gamma_k(G)$ is connected, then so is~$\Gamma_k(H)$.  This claim is not true for infinite groups.
}\end{rem}

\medskip

\section{Automorphisms of the rooted binary tree}\label{s:binary}

In this section, we introduce and discuss properties of the group $\AutT$ of
automorphisms of a binary tree.

\smallskip

\subsection{Definitions}
Let $\Tbin = \{0,1\}^*$ denote the rooted binary tree consisting of finite strings over the alphabet $\{0,1\}$, whose root is the empty string, where the children of the string $s$ are $s0$ and~$s1$. Define $\AutT$ to the group of automorphisms of this tree. Formally, $\AutT$ consists of length preserving bijections $g$ of $\Tbin$ such that for any $s,t \in \Tbin$, $g(st)$ begins with~$g(s)$. To avoid confusion with the bit $1$, we let $\i \in \AutT$ denote the identity element. Let $g\da_s$ denote the action of $g$ on tails of strings beginning with $s$. In other words, we define it to satisfy~$g(st) = g(s)g\da_s(t)$.

Define $a \in \AutT$ to be the automorphism which flips the first bit of~$s$. Formally, $a(0s) = 1s$ and~$a(1s) = 0s$ for all $s \in \T$. Clearly, every element of $\AutT$ either fixes $0$ and $1$ or swaps them. Let $g$ be an element that fixes them. Then $g(0s) = 0g\da_0(s)$ and $g(1s) = 1 g\da_1(s)$. In this case, we write $g$ in one of the following two forms, depending on which is more convenient:
\begin{align*}
g = (g\da_0, g\da_1)
\quad\text{or}\quad
g = \binom{g\da_0}{g\da_1}.
\end{align*}
On the other hand, suppose $g \in \AutT$ swaps $0$ and $1$. Then~$g = a(g\da_0, g\da_1) = (g\da_1, g\da_0)a$. Thus, we can write every element $g \in \AutT$ in the form $(h,k)a^\e$, for some $h,k \in \AutT$ and some $\e \in \{0,1\}$. Moreover $(h,k)a = a(k,h)$ for all $h,k \in \AutT$. In other words, we have $\AutT = (\AutT \times \AutT) \rtimes \zz_2$ where $\zz_2$ acts on $\AutT \times \AutT$ by swapping the two coordinates.

Let $s \in \Tbin$ be a given binary string. We say the \emph{stabilizer} of $s$ is the subgroup of $\AutT$ consisting of those elements $g \in \AutT$ which fix $s$:
\begin{align*}
\Stab(s) \deq \setcst[\big]{g \in \AutT}{ g(s) = s }.
\end{align*}
The $n$-th \emph{level stabilizer} is the subgroup of $\AutT$ consisting of those elements which fix the $n$-th level of $\Tbin$:
\begin{align*}
\Stab_n \deq \bigcap_{s \in \{0,1\}^n} \Stab(s).
\end{align*}
Let $g \in \Stab_n$. The \emph{$n$-support} of $g$ is
\begin{align*}
\supp_n(g) = \setcst[\big]{s \in \{0,1\}^n}{g \da_s \neq \i}.
\end{align*}
Finally, given $s \in \{0,1\}^n$, we define the \emph{rigid stabilizer} of $s$ to be the subgroup
\begin{align*}
\Rist(s) \deq \setcst[\big]{g \in \Stab_n}{\supp_n(g) \subseteq \{s\}}.
\end{align*}
In other words, $\Rist(s)$ consists of those elements of $\AutT$ which fix every string that does not begin with $s$.

For a subgroup $G$ of $\AutT$, define
\begin{align*}
\Stab_G(s) &
= G \cap \Stab(s)
\quad\text{and}\quad
\Rist_G(s)
= G \cap \Rist(s).
\end{align*}
Note that
\begin{align*}
\Rist(0s) &
= \{(g, \i) \mid g \in \Rist(s)\}
= \Rist(s) \times \{\i\},
\\
\text{and } \quad
\Rist(1s) &
= \{(\i,g) \mid g \in \Rist(s)\}
= \{\i\} \times \Rist(s).
\end{align*}

\subsection{Growth in subgroups of $\AutT$}

For distinct $s, s' \in \{0,1\}^m$, elements of $\Rist(s)$ and $\Rist(s')$ have disjoint $n$-support. We use Nielsen transformations to reach many of these elements. This implies we can find a large cubic set, which lets us construct many different generating $n$-tuples.

\begin{lemma} \label{lem:short_rist_fast}
Let $G < \AutT$ be finitely generated, and fix a generating $n$-tuple $S \in \Gamma_n(G)$. Suppose $G$ acts transitively on every level of $\Tbin$. Suppose there is a constant $\al$ such that for every $m\geq 1$, there is a string $s \in \{0,1\}^m$ and a nontrivial element $g \in \Rist_G(s)$ with~$\len(g) \leq \al 2^m$. Then $\Gamma_k(G,S)$ has exponential growth for every $k \geq n+2$
\end{lemma}
\begin{proof}
Given $m$, define $L = \{0,1\}^m$ and~$N = 2^m$. Fix $s \in L$ such that there is a nontrivial $g \in \Rist_G(s)$, satisfying~$\len(g) \leq \al N$.
Since $G$ acts transitively on $L$, we have that the Schreier graph $\Schr_S(G,L)$ is connected. Therefore, $\Schr_S(G,L)$ has a spanning tree $\CT$. Consider a depth-first traversal of $\CT$ with respect to the lexicographic order on $L$, starting at $s$. This is a path of length $2\abs{L}-2 < 2N$ which visits every element of~$L$. Suppose it visits them in the order $s_1, \dots, s_N$. For each $1 \leq i \leq N$, define $h_i$ to be the group element corresponding to the walk along this path from $s$ to $s_i$, so that $(s_1, \dots, s_N) = (h_1(s), \dots, h_N(s))$. Then we have $\len(h_2 h_1^{-1}) + \dots + \len(h_N h_{N-1}^{-1}) \leq 2N = 2^{m+1}$, and $\big(h_1(s), \dots, h_N(s)\big)$ is a permutation of the elements of~$L$.

Since $g \in \Rist_G(s)$, we have $h_i g h_i^{-1} \in \Rist_G\big(h_i(s)\big)$, for all $1 \leq i \leq N$. We claim that
\begin{align*}
(h_1 g h_1^{-1}, \dots, h_N g h_N^{-1})
\end{align*}
is a cubic $N$-tuple, i.e.
\begin{align*}
\size{\left\{\oaprod_{i=1 \dots N} (h_i g h_i)^{\e_i},\;\;\text{where}\;\;\e_i \in \{0,1\}\right\}} = 2^N.
\end{align*}
Indeed, the surjection $\phi: \{0,1\}^N \ra \cC(h_1 g h_1^{-1}, \dots h_N g h_N^{-1})$ given by
\begin{align*}
\phi(\e) \deq \oaprod_{i=1 \dots N} (h_i g h_i^{-1})^{\e_i}
\end{align*}
is also injective, since $\e_i = 1$ if and only if $s_i \in \supp_n \phi(\e)$. Hence $\size{\cC(h_1 g h_1^{-1}, \dots h_N g h_N^{-1})} = 2^N$, as desired.

Since $\len(g) \leq \al 2^m$, there is a path $\ga_1$ in $\Gamma_{n+1}(G,S)$ of length at most $\al 2^m$
\begin{align*}
S^{(1)} &
= (g_1, \dots, g_n, 1)
\pathto
(g_1, \dots, g_n, g)
= (g_1, \dots, g_n, h_1 g h_1^{-1}).
\end{align*}
Observe that the distance in $\Gamma_{n+1}(G,S)$ between $(g_1, \dots, g_n, h_i g h_i^{-1})$ and $(g_1, \dots, g_n, h_{i+1} g h_{i+1}^{-1})$ is at most~$2\len(h_{i+1} h_i^{-1})$. Since $\len(h_2 h_1^{-1}) + \dots + \len(h_N h_{N-1}^{-1}) \leq 2^{m+1}$, there is a path $\ga_2$ in $\Gamma_{n+1}$ of length at most $2^{m+2}$ which starts at $(g_1, \dots, g_n, g)$ and visits each $(g_1, \dots, g_n, h_i g h_i^{-1})$, in that order.

Composing $\ga_1$ and $\ga_2$, we see that there is a path in $\Gamma_{n+1}(G,S)$ of length at most $(\al+4)2^m$ which starts at $S^{(1)}$ and visits generating $(n+1)$-tuples containing $h_1 g h_1^{-1}, \dots, h_N g h_N^{-1}$, in that order. These elements of $G$ form a cubic $2^m$-tuple. Applying Lemma~\ref{lem:cubic_fast} with $k_m = 2^m$, then, tells us that $\Gamma_{k}(G,S)$ has exponential growth for all $k \geq n+2$.
\end{proof}

\begin{Remark}
We cannot replace $\len(g) \leq \al2^n$ in the hypotheses of this lemma with~$\len(g) \leq \al^n$ with some $\al > 2$. Roughly speaking, that would only let us reach a cubic $2^n$-tuple in $\al^n$ steps. Thus, we can only generate an $r^{1/d}$-cube in $B_\Gamma(S,r)$, where~$d = \log_2 \al$, which is not sufficient to guarantee exponential growth. We can, however, replace the assumption that $\len(g) \leq \al 2^n$ with the assumption that we can reach a generating $(n+1)$-tuple containing $g$ in $\al 2^n$ Nielsen moves.
\end{Remark}

\medskip

\section{The Grigorchuk group}\label{sec:grigorchuk}

\subsection{Definition}

The Grigorchuk group $\gg < \AutT$ is defined as $\gg = \gen{a,b,c,d}$, where $a$ flips the first bit of a string, and $b$, $c$, and $d$ are defined recursively by the relations \begin{align*}
b &\deq (a,c)\\
c &\deq (a,d)\\
d &\deq (\i,b).
\end{align*}
It is easy to check that $a^2 = b^2 = c^2 = d^2 = bcd = \i$. Thus, $\gg$ is actually generated by just three elements: $\gg = \gen{a, b, c}$.

Here is an explicit description of the action of these involutions on $\Tbin$.
\begin{align*}
d(1^n) &= 1^n
\\
d(1^n0s) &= \begin{cases}
1^n0s, & n \equiv 0 \pmod 3 \\
1^n0a(s), & n \equiv 1, 2 \pmod 3
\end{cases}
\end{align*}
In other words, $d$ changes at most one bit in a string -- the bit after the first~$0$. Specifically, $d$ flips that bit if and only if the number $n$ of $1$'s in the string up to that point is $1$~or~$2 \pmod 3$. Similarly, $c$ flips it when $n \equiv 0,2 \pmod 3$, and $b$ flips it when $n \equiv 0, 1 \pmod 3$.

\begin{thm}[Gigorchuk]\label{t:grig}
The group~$\gg$ has intermediate growth.
\end{thm}

The theorem was first proved by Grigorchuk in~\cite{grigorchuk_growth}
(see also~\cite{growthintro,de_la_harpe}).

\subsection{Connectivity of $\Gamma_n(\gg)$}  We prove the following result:

\begin{prop}\label{prop:conn-grig}
For each $n \geq 4$, the product replacement graph $\Gamma_n(\gg)$ is
connected (see also~$\S$\ref{ss:fin-amen}).\footnote{After this paper was
written, we learned that the proposition was independently derived
in~\cite{Myr}.}
\end{prop}

\begin{proof}
Fix $n \geq 4$. It is known that  $\gg / \Frat(\gg) \cong \zz_2^3$
(see~\cite{grig_frattini} and~\cite[$\S 6$]{solved_unsolved}).
The graph $\Gamma_n(\zz_2^3)$ is connected by Lemma~\ref{lem:(Z_2)^n-connected}.
Thus, by Lemma~\ref{lem:frattini_connected}, $\Gamma_n(\gg)$ is connected.
\end{proof}

\subsection{Exponential growth in $\Gamma_n(\gg)$}

The goal of this section is to prove the following result:

\begin{thm}\label{thm:grigorchuk_fast}
For each $n \geq 5$, the product replacement graph $\Gamma_n(\gg)$ of the Grigorchuk group has exponential growth.
\end{thm}

The proof is based on Lemma~\ref{lem:short_rist_fast}. Roughly, our strategy is to find an element $g$ of $\Rist_\gg(1^n)$ with length~$O(2^n)$. In $O(2^n)$ more steps, we conjugate $g$ to reach an element of $\Rist_\gg(s)$ for each $s$ on the same level of~$\Tbin$. Then we can construct every product of these conjugates in $O(2^n)$ steps. There are $2^{2^n}$ such products, which gives us exponential growth.

\begin{proof}[Proof of Theorem~\ref{thm:grigorchuk_fast}]
Fix $n \geq 5$. It is easy to check that $\gg$ acts transitively on the levels of $\Tbin$ (see e.g.~\cite[$\S$VIII]{de_la_harpe} or Lemma~\ref{lem:gg_transitive}, below). By Lemma~\ref{lem:short_rist_fast}, it suffices to show that for every $m \geq 0$, there is a nontrivial element of $\Rist(1^m)$ of length at most $2^{m + 4}$ with respect to the generating $3$-tuple $(a,b,c)$.

Define~$t_0 = abab$. Observe that $t_0^2(111)=110$, and therefore $t_0^2\neq \i$. We prove by induction on $m$ that there is a $t_m \in \gg$ of the form
\begin{align}\label{eq:t_m_def}
\tag{$*$}
t_m = \oaprod_{i=1 \dots N} abax_i,
\end{align}
where $N=2^m$, $x_i \in \{b,c,d\}$ for each~$1 \leq i \leq 2^m$, such that $t_m^2 \in \Rist_\gg(1^m)$ and $t_m\da_{1^m} = t_0$. The base case $m = 0$ is trivial.

Given $t_m$ and ($x_i$) related by (\ref{eq:t_m_def}), for each $0 \leq i \leq N$ we define $x'_i \in \{b,c,d\}$ by~$x'_i = (a^{\e_i}, x_i)$ where $\e_i \in \{0,1\}$. We define $t_{m+1}$ by applying the rewriting rules $a \mapsto aba$, $b \mapsto d$,  $c \mapsto b$, $d \mapsto c$ to~$t_m$. Then we have
\begin{align*}
t_{m+1}
&= \left[ \oaprod_{i=1 \dots N} (aba)d(aba)x'_i \right]
= \left[ \oaprod_{i=1 \dots N} \binom{c}{a}\binom{\i}{b}\binom{c}{a}\binom{a^{\e_i}}{x_i} \right]
= \binom{a^\e}{t_m}\,,
\end{align*}
Thus,
\begin{align*}
t_{m+1}^2 &
= (\i, t_m^2) \in \{\i\} \times \Rist(1^m)
= \Rist(1^{m+1}),
\end{align*}
and~$t_{m+1}^2 \da_{1^{m+1}} = t_m^2\da_{1^m} = t_0^2$.

Since $t_m^2 \da_{1^m} = t_0^2 \neq \i$, we can conclude that $t_m^2 \neq \i$. Hence, for every $m\geq 0$, we have that $t_m^2$ is a nontrivial element of $\Rist_\gg(1^m)$, with~$\len_{\gen{a,b,c}}(t_m^2) \leq 2 \len_{\gen{a,b,c,d}}(t_m^2) \leq 2^{m + 4}$, which concludes the proof.
\end{proof}

\medskip

\section{The generalized Grigorchuk groups}\label{s:gen_grig}

In this section, we use the same approach to analyze growth in the product replacement graph of $\gg_\omega$. The same techniques apply, but the technical details are more involved.

\subsection{Definition}

Let $\omega$ be an infinite string in the alphabet\footnote{The usual definition uses the alphabet $\{0,1,2\}$ but for our purposes it is more convenient to use $\{b, c, d\}$.} $\{b, c, d\}$. The \emph{generalized Grigorchuk group} $\gg_\omega$ is the group of automorphisms of $\{0,1\}^n$ given by $\gg_\omega = \gen{a, b_0, c_0, d_0}$. Here, the element $a$ flips the first digit of a string, and for each $x \in \{b,c,d\}$, the elements $x_n$ are defined recursively by
\begin{align*}
x_n \deq (a^\e, x_{n+1}),
\quad\text{ where }\;
\e =
\begin{cases}
0, & x = \omega_n \\
1, & \text{otherwise.}
\end{cases}
\end{align*}
For convenience, we write $b = b_0$, $c = c_0$, and $d = d_0$. As with $\gg$, we have $a^2 = b^2 = c^2 = d^2 = bcd = 1$.

As before, we give a more explicit description of the action of $\gg_\omega$ on $\Tbin$. Given $x \in \{b,c,d\}$ and $s \in \Tbin$,
\begin{align*}
x(1^n) &= 1^n, \text{ and }
\\
x(1^n0s) &=
\begin{cases}
1^n0s, & \omega_n = x \\
1^n0a(s), & \text{otherwise}.
\end{cases}
\end{align*}
Taking $\omega = dcbdcbdcbdcb\dots$ gives the usual Grigorchuk group. The following fact is well-known, but we include a proof here for completeness.

\begin{lemma} \label{lem:gg_transitive}
The generalized Grigorchuk group $\gg_\omega$ acts transitively on every level of~$\Tbin$.
\end{lemma}
\begin{proof}
We prove that $\gg_\omega$ acts transitively on the $n$-th level by induction on~$n$. This is trivial for $n = 0$, and true for $n = 1$ because~$a \in \gg_\omega$. For $n > 1$, note that it suffices to show that for each~$s \in \{0,1\}^n$, there is a~$g \in \gg_\omega$ such that $g(s) = 1^{n-2}00$. Consider~$s \in \{0,1\}^n$. We know that $s = s'd$, for some $s' \in \{0,1\}^{n-1}$ and~$d \in \{0,1\}$. By the induction hypothesis, $\gg_\omega$ acts transitively on~$\{0,1\}^{n-1}$. Thus there is a $g \in \gg_\omega$ with~$g(s') = 1^{n-2}0$. Then either $g(s) = 1^{n-2}00$ or~$g(s) = 1^{n-2}01$. In the latter case, there is an $x \in \{b,c,d\}$ such that $\omega_{n-2} \neq x$, and then $x(g(s)) = 1^{n-2}00$. In both cases, there is an $h \in \gg_\omega$ with~$h(s) = 1^{n-2}00$.
\end{proof}

\subsection{Exponential growth in $\Gamma_n(\gg_\omega)$}
To prove Theorem~\ref{thm:gg_fast}, we first need some lemmas about~$\gg_\omega$.
A standard computation shows that, under some weak assumptions on~$\omega$,
every element of $\gg_\omega$ has finite order. We will use the following
more specialized result.

\begin{lemma} \label{lem:gg_ad_order}
Suppose $\omega_{n-1} = d$. Then in $\gg_\omega$, we have $(a d_k)^{2^{n-k+1}} = \i$ for every~$0 \leq k < n$.
\end{lemma}

\begin{proof}
Since $\omega_{n-1} = d$, we have $d_{n-1} = (\i, d_{n})$ and~$ad_{n-1}a = (d_{n}, \i)$.
We prove the lemma by induction on~$j = n-k$. When $j = 1$, i.e.~$k = n-1$, we have
\begin{align*}
(a d_k)^4
= \left[(a d_{n-1} a) d_{n-1}\right]^2
= \left[\binom{d_{n}}{\i}  \binom{\i}{d_{n}}\right]^2
= \binom{d_{n}^2}{d_{n}^2}
= \i.
\end{align*}
When $j>1$, i.e.\ $k < n-1$, the induction hypothesis tells us~$(a d_{k+1})^{2^{j}} = \i$. Note that also $(d_{k+1})^{2^{j}} = \i$, since $d_{k+1}$ has order~$2$. Then, for some $\e \in \{0,1\}$, we have
\begin{align*}
(a d_k)^{2^{j+1}}
&= \left[(a d_k a) d_k\right]^{2^{j}}
= \left[\binom{d_{k+1}}{a^\e} \binom{a^\e}{d_{k+1}}\right]^{2^{j}}
= \binom{(a^\e d_{k+1})^{- 2^{j}}}{(a^\e d_{k+1})^{2^{j}}}
= \i.
\end{align*}
\end{proof}

\begin{lemma} \label{lem:gg_short_rist}
Suppose $\omega \in \{b,c,d\}^*$ is not eventually constant. Then for each $n\geq 0$, there is a nontrivial $t \in \Rist_{\gg_\omega}(1^n)$ with~$\len(t) \leq 2^{n+2}$.
\end{lemma}
\begin{proof}
This is trivial if~$n=0$. If $n>1$, then by relabeling $b$, $c$, and $d$ if necessary, we may assume~$\omega_{n-1} = d$.

By induction on $j = n-k$ we show that for every $0 \leq k \leq n$, there is a $t_k$ of the form
\begin{align*}
t_k = \oaprod_{i=1 \dots {2^{n-k}}} ax_i,
\end{align*}
where $x_i \in \{b_k, d_k\}$ for each $i$, and there is an odd number of $i$'s with $x_i = d_k$, such that~$t_k^2 \in \Rist(1^{n-k})$, and $t_k^2\neq \id$.

For $j = 0$, i.e.~$k = n$, we define~$t_n = ad_n$. We know that $d_n = (a^\e, d_{n+1})$ for some $\e \in \{0,1\}$, and therefore we have
\begin{align*}
t_n^2
= (ad_na) d_n
= \binom{d_{n+1}}{a^\e}\binom{a^\e}{d_{n+1}}
= \binom{d_{n+1}a^\e}{(d_{n+1} a^\e)^{-1}}.
\end{align*}
Since $\omega$ is not eventually constant, we know that there is an $m \geq n+1$ with $\omega_m \neq d$. Therefore we have
\begin{align*}
d_{n+1}(1^{m-n-1}00) = 1^{m-n-1}d_m(00) = 1^{m-n-1}0a(0) = 1^{m-n-1}01.
\end{align*}
Hence,~$d_{n+1} \neq \i$. It follows that $d_{n+1}a^\e$ is nontrivial whether $\e = 0$~or $\e=1$. Therefore,~$t_n^2 \neq \i$.

For $j=1$, i.e.~$k = n-1$, we define~$t_k = ab_{n-1}ad_{n-1}$. We have $\omega_k = d$, hence $d_k = (\i, d_{k+1})$ and~$b_k = (a, b_{k+1})$. Therefore, we have:
\begin{align*}
t_{n-1}^2
= \left[(ab_{n-1}a) d_{n-1}\right]^2
= \left[\binom{b_n}{a} \binom{\i}{d_n}\right]^2
= \binom{\i}{(ad_n)^2}
= \binom{\i}{t_n^2}.
\end{align*}

For $j>1$, i.e.~$k < n-1$, let $N = 2^{n-k-1}$. We have
\begin{align*}
t_{k+1} = \oaprod_{i=1 \dots N} ax_i
\end{align*}
from the previous step. For each $1 \leq i \leq N$, we know that $x_i = b_{k+1}$ or $d_{k+1}$, and we define
\begin{align*}
x'_i &
= \begin{cases}
b_k, & x_i = b_{k+1} \\
d_k, & x_i = d_{k+1}.
\end{cases}
\end{align*}
Then $x'_i=(a^{\e_i},x_i)$ for some $\e_i \in \{0,1\}$. We have three possibilities:

\begin{enumerate}[{Case} (i):]

\item \label{itm:b} $\omega_k = b$. Then $d_k = (a, d_{k+1})$ and~$b_k = (\i, b_{k+1})$. Define
\begin{align*}
t_k = \oaprod_{i=1 \dots N} ad_kax'_i.
\end{align*}
The product has an even number of terms, thus, we have not changed the parity of the number of $d_k$'s in the product, which implies it is still odd.

\begin{align*}
t_k
= \oaprod_{i=1 \dots N} (ad_ka)x'_i
= \oaprod_{i=1 \dots N} \binom{d_{k+1}}{a} \binom{a^{\e_i}}{x_i}
= \binom{\oaprod d_{k+1} a^{\e_i}}{t_{k+1}},
\end{align*}
where the final product runs over $i=1\dots N$. Observe that $\e_i = 1$ if and only if~$x_i = d_i$. There are an odd number of such~$i$, therefore $\oaprod d_{k+1} a^{\e_i}$ is a product containing an odd number of~$a$'s and an even number of~$d_{k+1}$'s. The elements $d_{k+1}$ and $a$ have order $2$, so the group $\gen{d_{k+1}, a}$ is a dihedral group in which they are both reflections. Hence, the product $\oaprod d_{k+1} a^{\e_i}$ is also a reflection in that dihedral group, and thus it has order 2. Therefore,~$t_k^2 = (\i, t_{k+1}^2)$.

\item \label{itm:d}  $\omega_k = d$. Then $d_k = (\i, d_{k+1})$ and~$b_k = (a, b_{k+1})$. Define \begin{align*}
t_k = \oaprod_{i=1 \dots N} ab_kax'_i,
\end{align*}
and argue as in case (\ref{itm:b}).

\item \label{itm:c} $\omega_k = c$. Then $d_k = (a, d_{k+1})$ and $b_k = (a, b_{k+1})$. Hence $x_i' = (a, x_i)$ for each~$0\leq i \leq N$. We can again define
\begin{align*}
t_k = \oaprod_{i=1 \dots N} ad_kax'_i.
\end{align*}
Then
\begin{align*}
t_k
= \oaprod_{i=1 \dots N} (ad_ka) x'_i
= \oaprod_{i=1 \dots N} \binom{d_{k+1}}{a} \binom{a}{x_i}
= \binom{(d_{k+1} a)^N}{t_{k+1}}.
\end{align*}
By Lemma~\ref{lem:gg_ad_order}, we have $(d_{k+1} a)^{2N} = (d_{k+1} a)^{2^{n-k}} = \i$. Hence,~$t_k^2 = (\i, t_{k+1}^2)$.
\end{enumerate}

In all three cases, $t_k^2 = (\i, t_{k+1}^2)$. It follows that $t_k^2$ is nontrivial and
\begin{align*}
t_k^2 \in \{\i\} \times \Rist(1^{n-k-1}) &
= \Rist(1^{n-k}).
\end{align*}

Thus, we have a nontrivial $t_0^2 \in \Rist(1^n) \cap \gg_\omega$, with~$\len(t_0^2) \leq 2^{n+2}$, as desired.
\end{proof}

\begin{proof}[Proof of Theorem~\ref{thm:gg_fast}]
It is known that $\gg_\omega/\Frat(\gg_\omega) \cong \zz_2^k$ for some $k \leq 3$ \cite{grig_frattini} (see also \cite[$\S 6$]{solved_unsolved}). Recall from Lemma~\ref{lem:(Z_2)^n-connected} that $\Gamma_n(\zz_2^k)$ is connected. Lemma~\ref{lem:frattini_connected} tells us that $\Gamma_n(\gg_\omega)$ is connected for each~$n \geq 4$.

Assume that $\omega$ is eventually constant. Then it is not hard to check that $\gg_\omega$ has polynomial growth. In fact, $G_\omega$ is virtually abelian \cite[$\S 2$]{solved_unsolved}. It follows that $G_\omega$ has an element of infinite order.
The group $\gg_\omega$ is generated by three elements, $\gg_\omega = \gen{a, b, c}$. By Lemma~\ref{lem:infinite_order_fast}, this implies that the product replacement graph $\Gamma_n(\gg_\omega)$ has exponential growth for each~$n \geq 5$.

Otherwise, if $\omega$ is not eventually constant, for every $m\geq 0$, Lemma~\ref{lem:gg_short_rist} gives a nontrival $t \in \Rist_{\gg_\omega}(1^m)$ of length at most~$2^{m+2}$. Since $\gg_\omega$ acts transitively on the levels of $\Tbin$, we can apply Lemma~\ref{lem:short_rist_fast} to conclude that $\Gamma_6(\gg_\omega)$ has exponential growth from~$(a, b, c, d, 1, 1)$.

Moreover, note that the group $\gg_\omega$ is generated by $(a,b,c)$, and rewriting $t$ as a word in these generators at most doubles its length. Thus, we also have that $\Gamma_5(\gg_\omega)$ has exponential growth from~$(a,b,c,1,1)$. It follows that $\Gamma_n(\gg_\omega)$ has exponential growth for each~$n \geq 5$.
\end{proof}

\medskip

\section{Final remarks}\label{s:fin}

%\subsection{}\label{ss:fin-aut}

%\subsection{}\label{ss:fin-amen1}

\subsection{}\label{ss:fin-amen}
There are several other directions in which our Theorem~\ref{thm:grigorchuk_fast}
can be extended.  First, there is the problem of smaller~$k$: we believe that
that $\Gamma_3(\gg)$ is connected (cf.~Lemma~\ref{lem:(Z_2)^n-connected}
and Proposition~\ref{prop:conn-grig}).\footnote{See also Corollary~1.2 and
Question~1 in~\cite{Myr}.}
Moreover, it is conceivable that both $\Gamma_3(\gg)$ and $\Gamma_4(\gg)$
have exponential growth, the cases missing from Theorem~\ref{thm:grigorchuk_fast}.

Similarly, in case Conjecture~\ref{conj:grig-nonam} proves too difficult, there is a
weaker and perhaps more accessible open problem.

\begin{conj}
The nearest neighbor random walk on $\Gamma_k(\gg)$ has positive speed, for all $k\ge 5$.
\end{conj}

The speed of r.w.~is defined as the limit of $\ee[\dist(t)/t]$ as $t\to \infty$,
where \ts $\dist(t)$ is the distance of the r.w.~after~$t$ steps, from the starting vertex.
It is known that non-amenable graphs have positive speed, but so do some amenable graphs,
such as the standard Cayley graph of the lamplighter group $\zz_2 \wr \zz^3$
(see e.g.~\cite{Pete, Woess}).  We believe it might be possible to extend our approach
to establish the positive speed of r.w.~on $\Gamma_k(\gg)$, and we intend to return
to this problem.

\subsection{}\label{subsec:automorphisms_fast}
It is easy to check that if $\Aut(G)$ has a finitely generated subgroup $A = \gen{\phi_1, \dots, \phi_k}$ of exponential growth, then for any $S = (s_1, \dots, s_n) \in \Gamma_n(G)$, the graph $\Gamma_{2n}(G,S)$ has exponential growth. Indeed, for any $\psi \in A$ with $\len(\phi) \leq r$, we have that $(\psi(s_1), \dots, \psi(s_n), 1, \dots, 1)$ is within $O(r)$ Nielsen moves of $(s_1, \dots, s_n, 1, \dots 1)$, since there there are paths of bounded length
\begin{align*}
\big( \psi(s_1), \dots, \psi(s_n), 1, \dots, 1 \big)
&\pathto
\big( \psi(s_1), \dots, \psi(s_n), \psi(\phi_i(s_1)), \dots, \psi(\phi_i(s_n))\big)
\\
&\pathto
\big( 1, \dots, 1, \psi(\phi_i(s_1)), \dots, \psi(\phi_i(s_n))\big)
\\
&\pathto
\big(\psi(\phi_i(s_1)), \dots, \psi(\phi_i(s_n)), 1, \dots, 1\big).
\end{align*}
Unfortunately, this assumption does not hold for the Grigorchuk group (see \cite{grigorchuk_aut}).

\subsection{}
It is tempting to try to prove that the Grigorchuk group $\gg$ has exponential Nielsen growth by extending the observation in Subsection~\ref{subsec:automorphisms_fast} to endomorphisms rather than automorphisms. In order to show that $\gg$ has exponential Nielsen growth, it is enough to show that some subgroup $K$ of $\gg$ has exponential Nielsen growth. Indeed, there is a finite index subgroup $K = \gen{t,v,w}$ of $\gg$ which satisfies $K \times K \subset K$ (see~\cite[$\S$VIII.30]{de_la_harpe}). It follows that $K$ has a free monoid of endomorphisms, generated by $\phi_0$ and $\phi_1$. Then, for any word $\psi$ in $\phi_0$ and $\phi_1$ with $\len(\psi) \leq r$, one might hope that $(t,v,w, \psi(t),\psi(v),\psi(w), 1,1,1)$ is within $O(r)$ Nielsen moves of $(t,v,w, 1,1,1, 1,1,1)$. This would be the case if the following paths had bounded length:
\begin{align*}
\big(t,v,w, \psi(t),\psi(v),\psi(w), 1,1,1\big)
&\pathto
\big(t,v,w, \psi(t),\psi(v),\psi(w), \psi(\phi_i(t)),\psi(\phi_i(v)),\psi(\phi_i(w))\big)
\\
&\pathto[\star]
\big(t,v,w, 1,1,1, \psi(\phi_i(t)),\psi(\phi_i(v)),\psi(\phi_i(w))\big)
\\
&\pathto
\big(t,v,w, \psi(\phi_i(t)),\psi(\phi_i(v)),\psi(\phi_i(w)), 1,1,1\big).
\end{align*}
However, it is not clear that the ``cleanup'' step marked with $^\star$ can be done in a bounded number of Nielsen moves. Hence, unfortunately, the fact that $\gg$ has a finite index subgroup with a free monoid of endomorphisms is not enough to conclude that $\gg$ has exponential Nielsen growth.

%\subsection{}\note{move somwhere else}
%
%\begin{prop}\label{prop:exponential_fast}
%Let $G$ be a finitely generated group, and let $S \in \Gamma_n(G)$. Suppose $G$ has exponential growth. Then $\Gamma_m(G,S)$ has exponential growth for every $m \geq n+1$.
%\end{prop}
%
%\begin{prop}\label{prop:polynomial_fast}
%Let $G$ be a finitely generated group, and let~$S \in \Gamma_n(G)$. Suppose $G$ is infinite and has polynomial growth. Then $\Gamma_m(G,S)$ has exponential growth for every $m \geq n+2$.
%\end{prop}
%
%\begin{cor}\label{thm:virtually_solvable_non-amenable}
%Let $G$ be a finitely generated infinite virtually solvable group. Then $\Gamma_m(G,S)$ has exponential growth,
%for every $m \geq n+2$.
%\end{cor}

%\subsection{}

%\subsection{}\label{ss:fin-gen}

%\subsection{}\label{ss:fin-quo}

\subsection{}\label{ss:fin-conn}
The connectivity of product replacement graphs is delicate already for finite groups.
For example, Dunwoody showed in~\cite{dunwoody_nielsen}, that if $G$ is a finite
solvable group with $d$ generators, then $\Gamma_k(G)$ is connected,
for every $k > d$ (see also~\cite{whatprp}). This property is conjectured to hold
for all finite groups, but fails for infinite groups, even for metabelian groups
(see~\cite{whatprp} and references therein).

As of now, is unknown whether for any finitely generated group $G$, graphs $\Gamma_k(G)$
are connected for all sufficiently large~$k$. It is not even known that if $\Gamma_k(G)$
is connected then $\Gamma_{k+1}(G)$ is connected. The difficulty arises from the
possibility that $\Gamma_{k+1}(G)$ has a connected component which consists of
non-redundant generating $(k+1)$-tuples.
However, it is not hard to check that in $\Gamma_{2d}(G)$ every element of the form
$(g_1, \dots, g_d, 1, \dots, 1)$ lies in the same connected component, which we may
call $\Gamma_{2d}^\star(G)$. Then if we know that some connected component of
$\Gamma_d(G)$ has exponential growth, we know that $\Gamma_{2d}^\star(G)$ has
exponential growth.

%\subsection{}\label{ss:fin-virt}

\subsection{}\label{ss:fin-finite}
Finally, let us mention that the notion of exponential Nielsen growth may be applicable to
sequences of finite groups, which stabilize in a certain sense.  Proving such a result would
be a step towards proving expansion of product replacement graphs of general finite
groups (see~\cite{whatprp,Pak2}).  We refer to~\cite{Black} for the notion of
growth of finite groups, and to~\cite{Ell} for a recent conceptual approach.

\vskip.7cm

\noindent
{\bf Acknowledgements.} \ts The authors are grateful to Tatiana Nagnibeda who brought
to our attention a question on connectivity and exponential growth of $\Gamma_k(\gg)$,
and help with the references.  We are also very thankful to Slava Grigorchuk and
Martin Kassabov for interesting conversations on groups of intermediate growth, and to
Yehuda Shalom for helpful remarks on uniform growth.  The second author was partially
supported by the BSF and the~NSF.

\end{document}